\documentclass[11pt,thmsa]{article}%
\usepackage{amsmath}
\usepackage{amsfonts}
\usepackage{amssymb}
\usepackage{graphicx}%
\newtheorem{theorem}{Theorem}[section]

\newtheorem{lemma}[theorem]{Lemma}

\newtheorem{proposition}[theorem]{Proposition}

\newenvironment{proof}[1][Proof]{\noindent\textbf{#1.} }{\ \rule{0.5em}{0.5em}}

\begin{document}
\title{Clifford-Wolf homogeneous Finsler metrics on spheres\footnote{Supported by NSFC (no. 11271198, 11221091,11271216) and SRFDP of China}}
\author{Ming Xu$^1$ and Shaoqiang Deng$^2$\thanks{S. Deng is the corresponding author. E-mail: dengsq@nankai.edu.cn}\\
\\
$^1$College of Mathematical Sciences\\Tianjin Normal University \\Tianjin 300387\\People's  Republic of China\\
\\$^2$School of Mathematical Sciences and LPMC\\
Nankai University\\
Tianjin 300071\\
People's Republic of China}
\date{}
\maketitle
\begin{abstract}
An isometry of a Finsler space is called Clifford-Wolf translation (CW-translation) if it moves all points the same distance. A Finsler space $(M, F)$ is called
  Clifford-Wolf homogeneous (CW-homogeneous) if for any $x, y\in M$ there is a CW-translation $\sigma$ such that $\sigma (x)=y$. We prove that if $F$ is a homogeneous Finsler metric on the sphere $S^n$ such that $(S^n, F)$ is CW-homogeneous, then $F$ must be a Randers metric. This gives a complete classification of CW-homogeneous Finsler metrics on spheres.


\end{abstract}

\section{Introduction}
It is our goal in this article to give a complete classification of Clifford-Wolf homogeneous Finsler metrics on spheres. Recall that an isometry $\sigma$ of
a connected Finsler space $(M, F)$ is called a Clifford-Wolf translation (simply CW-translation) if $d (x, \sigma (x))$ is a constant function on $M$.
A Finsler space $(M, F)$ is called Clifford-Wolf homogeneous (simply CW-homogeneous) if for any $x, y\in M$ there is an CW-isometry $\sigma$ such that $\sigma (x)=y$. More generally, a Finsler space $(M, F)$  is called restrictively CW-homogeneous if for any $x\in M$, there is an open neighborhood $V$ of $x$ such that for any $y\in V$ there is a CW-translation $\sigma$ of $(M,Q)$ such that $\sigma(x)=y$. According to the study of  V.N. Berestovskii and Yu.G. Nikonorov in \cite{BN09}, any restrictively CW-homogeneous Riemannian manifold must be locally symmetric. Based on this, the authors of \cite{BN09} obtained a complete classification of all connected simply connected CW-homogeneous Riemannian manifolds. The complete list consists of   compact Lie groups with bi-invariant Riemannian metrics, the odd-dimensional spheres with standard metrics, the symmetric spaces $\mathrm{SU}(2n+2)/\mathrm{Sp}(n)$ with the standard symmetric metrics, and the products of the above Riemannian manifolds.

In our previous considerations (\cite{DM1}-\cite{DM4}),  we have initiated the study of CW-translations of Finsler spaces and have found some new phenomena. In particular, there exist some
non-Riemannian Randers metrics on spheres as well as compact Lie groups that are CW-homogeneous but essentially not locally symmetric. This raises the interesting problem to classify CW-homogeneous Finsler spaces.

In this article we prove that any CW-homogeneous Finsler metric on a sphere must be a Randers metric. More precisely, we have
\begin{theorem}\label{main}
Let $F$ be a homogeneous Finsler metric on a sphere $S^n$ such that $(S^n, F)$ is a CW-homogeneous Finsler space. Then $F$ must be a Randers metric
\end{theorem}
Combining Theorem \ref{main} with the classification of CW-homogeneous Randers metrics in \cite{DM4},  we obtain a complete classification of CW-homogeneous Finsler metrics on spheres. Taking into account of the classification of CW-homogeneous Riemannian manifolds of \cite{BN09}, we can say that this fact is an important step to a complete classification of CW-homogeneous Finsler spaces.

In Section 2, we recall some definitions and fundamental results. Section 3 is devoted to giving a theme for the proof of the main theorem. In Section 4, we study triality for Cayley numbers and homogeneous spheres, and obtain a lemma which is useful for proving the main theorem. In Section 5, we complete the proof of Theorem \ref{main}.

\section{Preliminaries}

A Finsler metric $F$ on a manifold $M$ is a continuous function $F:TM\rightarrow \mathbb{R}^+$
which is smooth on $TM\backslash 0$, and satisfies the following conditions for any induced local  coordinates $(x^i,y^j)$
on $TM$, where (by some abuse of notation) the  $x^i$'s are the local coordinates on $M$:
\begin{description}
\item{(1)}\quad $F(x,y)>0$ for any $y\neq 0$;
\item{(2)}\quad $F(x,\lambda y) = \lambda F(x,y)$ for any  $\lambda > 0$;
\item{(3)}\quad The Hessian matrix $g_{ij}(x,y)(v)=\frac{1}{2}[F^2]_{y^i y^j}(v)$ is positive definite for all nonzero tangent $v$ in the domain of the induced chart on $TM$.
\end{description}
If $(M,g)$ is a Riemannian manifold and $F$ is the norm associated with $g$, then $(M, F)$ is clearly a   Finsler space.  As another example, let $\alpha$ be a Riemannian metric on a manifold $M$  and $\beta$ be a one-form with $||\beta||_{\alpha}=\mathop{\sup}\limits_{v\in T_x(M)- 0}\frac{\beta(v)}{\sqrt{\alpha(v,v)}}<1$. Then the  function $F=\alpha+\beta$ on $TM\backslash 0$  defines a Finsler metric on $M$, called a Randers metric. Randers metrics are
 the most simple and important Finsler metrics. It is obvious that the Randers metric $F$ as defined above is Riemannian if and only if $\beta=0$.

The Finsler metric $F$ naturally defines a (generally nonsymmetric) distance function on $M$, which we denote by $d$. The concept of Clifford-Wolf translation can be naturally extended to Finsler manifolds, i.e., a Clifford-Wolf translation (CW-translation) $\rho$ is an isometry of $(M,F)$ such that the function which sends a point $p$ of $M$ to the real number $d(p, \rho (p))$
is constant.

A Finsler manifold $(M,F)$ is called homogeneous if the isometry group $I(M,F)$ (which is a Lie
group) acts transitively on $M$. It is  called CW-homogeneous if the transitivity can be
achieved by the set of CW-translations. There is a slightly weaker version of CW-homogeneity, called
the restrictive CW-homogeneity. A Finsler manifold $(M,F)$ is called restrictively CW-homogeneous
if for any $x\in M$, there is a neighborhood $U$ of $x$ such that any point of $ U$ can be obtained as the image of  $x$ by a
CW-translation.

The study of CW-translations and CW-homogeneity in Finsler geometry was initiated in \cite{DM1}, by the interrelations between one-parameter local semi-groups of CW-translations and Killing vector fields of constant length (KFCLs).

Restrictive CW-homogeneity can be equivalently described as certain properties of KFCLs, namely, we have

\begin{proposition}
A homogeneous Finsler manifold $(M,F)$ is restrictively CW-homogeneous if and only if
any tangent vector can be extended to a KFCL for $(M,F)$.
\end{proposition}

Let $M=G/H$ be a homogeneous manifold such that $G$ acts almost effectively on $M$. Suppose $F$ is a $G$-invariant Finsler metric on $M$.  Denote the Lie algebras of $G$ and $H$
by $\mathfrak{g}$ and $\mathfrak{h}$, respectively. The tangent space $T_x(M)$ can be identified with
$\mathfrak{m}=\mathfrak{g}/\mathfrak{h}$. The natural projection from $\mathfrak{g}$ to $\mathfrak{m}$
will be denoted as $\mathrm{pr}$.
As a subalgebra of $\mathrm{Lie}\,(I(M,F))$, any vector $X$ in $\mathfrak{g}$ defines
a Killing vector field of $(M,F)$ (the fundamental vector field generated by $X$), which will also be denoted as $X$. The following easy lemma can be used to determine whether a Killing field $X$ is of constant length.

\begin{lemma} A Killing vector field $X\in \mathfrak{g}$ is of constant length $1$ if and only if
$\mathrm{pr}(\mathcal{O}_X)$ is contained in the indicatrix of $F$ in $\mathfrak{m}=T_x(M)$, where
$\mathcal{O}_X$ is the $\mathrm{Ad}(G)$-orbit of $X$.
\end{lemma}

Given any $p\in M$, we can write $p=gH $ for some $g\in G$. Note that $F$ is $G$-invariant. Hence the map $L_g: g_1H\to gg_1H$ is an isometry of $M$.
The inverse map $(L_g)^{-1}=L_{g^{-1}}$ is then an isometry of $M$ sending $p$ to the origin $x$. Suppose $X$ is a Killing vecotr field of $(M, F)$. Then we have
\begin{eqnarray*}
d(L_{g^{-1}})(X_p) &=&d(L_{g^{-1}})\frac{d}{dt}\{(\exp tX)gH\}\left|_{t=0}\right.\\
&=&\frac{d}{dt}\{g^{-1}(\exp tX)gH\}\left|_{t=0}\right.\\
&=&\frac{d}{dt}\{\exp (t\mathrm{Ad}(g^{-1})X)H\}\left|_{t=0}\right.\\
&=& \mathrm{pr}(\mathrm{Ad}(g^{-1})X).
\end{eqnarray*}
This means that the set ${\rm pr}(\mathcal{O}_X)$ exhausts the values of the Killing vector field $X$ at all points
when isometrically pulled back to $T_x(M)=\mathfrak{m}$.

\section{Theme for the proof of the main theorem: CW-translations and KFCLs}
\label{first}

The proof of the main theorem (Theorem \ref{main}) will be completed through a case by case discussion. Notice that for any Finsler manifold
 $(M,F)$, we can define a Riemannian metric $Q$ on $M$ by taking the average of the Hessian $(g_{ij})$ on the indicatrix of $F$. This implies that the group $I(M,F)$ is contained in $I(M,Q)$, and the connected component $I_0(M,F)$ is
 contained in $I_0(M,Q)$. Moreover, the Riemannian manifold $(M,Q)$ will be homogeneous if $(M,F)$ is homogeneous.

Now suppose $(M,F)$ is a  homogeneous Finslerian sphere. Then we can write $M=G/H$, where $G$ is a closed connected subgroup of $I_0(M,F)
\subset I_0(M,Q)$, with $Q$ being a Riemannian metric on $M$. By the results of \cite{MS43},  $M=G/H$ must be one of the following:
\begin{description}
\item{(1)}\quad $S^n=\mathrm{SO}(n+1)/\mathrm{SO}(n)$;
\item{(2)}\quad $S^{2n+1}=\mathrm{U}(n+1)/\mathrm{U}(n)$;
\item{(3)}\quad $S^{2n+1}=\mathrm{SU}(n+1)/\mathrm{SU}(n)$;
\item{(4)}\quad $S^{4n+3}=\mathrm{Sp}(n+1)/\mathrm{Sp}(n)$;
\item{(5)}\quad $S^{4n+3}=\mathrm{Sp}(n+1)S^1/\mathrm{Sp}(n)S^1$;
\item{(6)}\quad $S^{4n+3}=\mathrm{Sp}(n+1)\mathrm{Sp}(1)/\mathrm{Sp}(n)\mathrm{Sp}(1)$;
\item{(7)}\quad $S^6=G_2/\mathrm{SU}(3)$;
\item{(8)}\quad $S^7=\mathrm{Spin}(7)/G_2$;
\item{(9)}\quad $S^{15}=\mathrm{Spin}(9)/\mathrm{Spin}(7)$.
\end{description}

Now we start the proof of the main theorem.
Assume that $(M,F)$ is a Finslerian CW-homogeneous sphere with a connected transitive group of isometries,  $G\subset I_0(M,F)$.
If $G$ is $\mathrm{SO}(n+1)$ as in case (1),
then $F$ is obviously Riemannian. By the results of \cite{BN09}, $M$ must be odd-dimensional. In the following we consider the case  when $G\ne \mathrm{SO}(n+1)$.

The cases that $G=\mathrm{SU}(n+1)$, $\mathrm{U}(n+1)$, $\mathrm{Sp}(n+1)$, $\mathrm{Sp}(n+1)\mathrm{U}(1)$ or $\mathrm{Sp}(n+1)\mathrm{Sp}(1)$ have essentially been discussed earlier
in our previous papers (\cite{DM1}-\cite{DM4}). Their Lie algebras are the classical ones, possibly directly summed with a one-dimensional center or a copy of  summand that is isomorphic to $\mathfrak{sp}(1)$. The projections from $\mathfrak{g}$ to $\mathfrak{m}$ are their left multiplication on the column
vector $(0,0,\ldots,0,1)^t$. When $G=\mathrm{Sp}(n+1)\mathrm{U}(1)$ or $\mathrm{Sp}(n+1)\mathrm{Sp}(1)$, the $\mathfrak{u}(1)$ or $\mathfrak{sp}(1)$ factor has a scalar action on the column vectors from the right. Now let us recall the main results in these cases.

When $G=\mathrm{U}(n+1)$ or $\mathrm{Sp}(n+1)\mathrm{U}(1)$, the actions of its center provide $M$ with an $S^1$-principal bundle structure.
Any vector which is not tangent to a fiber can be extended to a Killing vector field of constant length
which is not contained in the center.
If a Killing vector field $X\in \mathfrak{u}(n+1)$ is of constant length $1$ on
$M=S^{2n+1}=\mathrm{U}(n+1)/\mathrm{U}(n)$, and if $X$ is not a central vector, then its $\mathrm{Ad}(\mathrm{U}(n+1))$-orbit
contains a Weyl group orbit of diagonal matrices, whose $\mathrm{pr}$-images are all contained in
a one-dimensional subspace coordinated by the eigenvalues of $X$. So the eigenvalues of $X$
 either are all the same, or  have two
different values with different signs. It is obvious that the eigenvalues of $X$ cannot be all the same, since $X$ is not in the center.
Therefore $X$ has two different eigenvalues with different signs. As  calculated in \cite{DM1}, the $\mathrm{Ad}(\mathrm{U}(n+1))$-orbit of $X$
is an ellipsoid. This implies that $F$ is a Randers metric. The discussion for $X\in \mathfrak{su}(n+1)$ is similar,
and the full connected isometry group is $\mathrm{U}(n+1)$ or $\mathrm{SO}(2n+2)$.

If $X\in \mathfrak{sp}(n+1)\oplus \mathbf{i}\mathbb{R}$
is a Killing vector field of constant length $1$ on $M=S^{4n+3}=\mathrm{Sp}(n+1)\mathrm{U}(1)/\mathrm{Sp}(n)\mathrm{U}(1)$, which is not contained
in the center,
we can  similarly find a subset of $\mathcal{O}_X$, in which each element can be represented as
$(\mathbf{i}X',a\mathbf{i})$, such that  all the $X'$s are nonzero real diagonal matrices, and are different
by Weyl group actions (i.e., permutation and sign changes for the diagonal entries), and $a$ is a real number
which  depends only on the $\mathbf{i}\mathbb{R}$ factor of $X$. They are all projected into a line
by $\mathrm{pr}$. Similarly as in the above case, one can show that all diagonal entries of $X$ have the same absolute values. We
have also proved  in \cite{DM4} that the set $\mathrm{pr}\mathcal{O}_X$ must be an ellipsoid.
Thus  $F$ is a Randers metric. Meanwhile, one easily shows   that in this case the connected isometry group is in fact $\mathrm{U}(2n+2)$,
or $\mathrm{SO}(4n+4)$ when $a$ vanishes.

By similar discussions,
we have also shown that $X\in \mathfrak{sp}(n+1)$ can not give a nonzero Killing vector field of constant length, see \cite{DM4}.

The discussion for the case $G=\mathrm{Sp}(n+1)\mathrm{U}(1)$ can also be applied to $G=\mathrm{Sp}(n+1)\mathrm{Sp}(1)$.
Let $X\in \mathfrak{sp}(n+1)\oplus \mathfrak{sp}(1)$ be a Killing vector field of constant length $1$ which is not contained in the $\mathfrak{sp}(1)$
factor.
Then in the $\mathrm{Ad}(G)$-orbit
of $X$, we can find a subset in which each element can be written as $(\mathbf{i}\mathbf{i}X',a\mathbf{i})$, where
$X'$ is a real diagonal matrix and $a$ is a real number. The $X'$s are only different by permutations and sign changes of their diagonal entries, and $a$ can change sign. All these points
are mapped to a one-dimensional subspace in $\mathfrak{m}$, which can contain at most two different points. It is not hard to see that $a$ must be zero. Then $\mathrm{pr}(\mathcal{O}_X)$ is round sphere
(with respect to a bi-invariant metric) centered at $0$. This indicates that $F$ is a bi-invariant Riemannian metric.

The case (7) that $G=G_2$  is trivial. The rank of $G=G_2$ is same as
that of $H=SU(3)$. Hence any Killing vector field $X\in\mathfrak{g}$ is conjugate with
some vector in $\mathfrak{h}$. Thus it must have a zero and can not
be of constant length.

In summarizing, we have proved the following proposition.

\begin{proposition}\label{prop0}
Let $F$ be a $G$-invariant Finsler metric on a sphere $S^n=G/H$, where $G=\mathrm{SO}(n+1)$, $\mathrm{U}(m+1)$ $(n=2m+1)$, $\mathrm{SU}(m+1)$ $(n=2m=1)$,
$\mathrm{Sp}(m+1)$ $(n=4m+3)$, $\mathrm{Sp}(m+1)\mathrm{U}(1)$ $(n=4m+3)$ or $\mathrm{Sp}(m+1)\mathrm{Sp}(1)$ $(n=4m+3)$. If $F$ is CW-homogeneous, then $F$ must be a Randers metric.
\end{proposition}

The only cases left are the case (8) that $S^7=\mathrm{Spin}(7)/G_2$ and
(9) that $S^{15}=\mathrm{Spin}(9)/\mathrm{Spin}(7)$, which will be dealt with by  applying some results on triality for Cayley numbers.
The discussion for these two cases will be postponed to Section 5.

\section{Triality for Cayley numbers and homogeneous spheres}

The algebra of Cayley numbers, denoted as $\mathcal{O}$, can be identified with $\mathbb{H}\oplus\mathbb{H}$ as a real vector space. The multiplication is then
defined by $(q_1,q_2)(s_1,s_2)=(q_1 s_1-\overline{s_2}q_2, s_2 q_1+ q_2\overline{s_1})$.
Denote $\overline{(q_1,q_2)}=(\overline{q_1},-q_2)$, $\mathrm{Re}(q_1,q_2)=\mathrm{Re} (q_1)$ and
$\mathrm{Im}(q_1,q_2)=(\mathrm{Im} (q_1),q_2)$. The inner product on $\mathbb{R}^8=\mathcal{O}$ is then defined as
$\langle(q_1,q_2),(s_1,s_2)\rangle=\mathrm{Re}(\overline{(q_1,q_2)}(s_1,s_2))$.

Any left or right multiplications by an element of unit length preserves the inner product
on $\mathbb{O}$. More generally, we have the following lemma.

\begin{lemma}
$\langle xy,xz\rangle=|x|^2\langle y,z\rangle$, $\langle yx,zx\rangle=|x|^2\langle y,z\rangle$,
$\forall x,y,z\in\mathbb{O}$.
\end{lemma}

Denote the left and right multiplication by $w\in\mathbb{O}$ as $L_w$ and $R_w$, respectively. We now recall some easy and well known facts.

\begin{lemma}
For any
$x,y,w\in\mathbb{O}$, we have
\begin{description}
\item{\rm (1)}\quad $\langle x,wy\rangle=\langle \overline{w}x,y\rangle$, i.e., $L_w^t=L_{\overline{w}}$.
Similarly, we have $R_{w}^t=R_{\overline{w}}$.
\item{\rm (2)}\quad  $\langle x,y\rangle=x\overline{y}=-y\overline{x}$.
\item{\rm (3)}\quad  $(xy)\overline{y}=x|y|^2$,  $\overline{x}(xy)=|x|^2 y$.
\item{\rm (4)}\quad Suppose  $x\neq 0$ and denote $x^{-1}=\overline{x}/|x|^2$. Then  $xw=y$ if and only if  $w=x^{-1}y$. Similarly,  $wx=y$ if and only if  $w=yx^{-1}$.
\item{\rm (5)}\quad $L_x L_{\overline y}+L_y L_{\overline{x}}=R_x R_{\overline{y}}+ R_y R_{\overline{x}}=2\langle x,y\rangle\mathrm\,{id}$.
\end{description}
\end{lemma}

Though the algebra of Cayley numbers is not associative, there are still some partial associativities.
For example, we have the following lemma.

\begin{lemma} Any subalgebra of $\mathbb{O}$ generated by two elements is associative. In particular, the products of the elements within the subalgebra $\mathbb{R}$, $\mathbb{C}$ or
$\mathbb{H}$ are associative.
\end{lemma}

The following lemma is the well known Moufang equalities.

\begin{lemma}
For any $x,y.z\in\mathbb{O}$, we have $(xyx)z=x(y(xz))$, $z(xyx)=((zx)y)x$ and $(xy)(zx)=x(yz)x$.
\end{lemma}

A triality triple is defined as a triple $(A,B,C)\in \mathrm{SO}(8)\times \mathrm{SO}(8)\times \mathrm{SO}(8)$, such that
$A(x)B(y)=C(xy)$ for any $x,y\in \mathbb{O}$. The set of all
triality triples is a connected Lie group. The key observation below implies that this Lie group is actually isomorphic to
$\mathrm{Spin}(8)$.

The triality for Cayler numbers refers to the following observations.

\begin{proposition} \label{triality_proposition}
For any $A\in \mathrm{SO}(8)$, there exist exactly two triality triples with
the first entry equal to $A$, i.e., of the form $(A,\pm B,\pm C)$.
The statement is also true for the second and third entries.
\end{proposition}

The projection from $\mathrm{Spin}(8)$ to $\mathrm{SO}(8)$ is specified as $\pi(A,B,C)=A$, though there are
two others implied by Proposition \ref{triality_proposition}. A subgroup $\mathrm{Spin}(7)$ in this
$\mathrm{Spin}(8)$ is defined by the condition $A(1)=1$, or equivalently $B=C$. A subgroup $G_2$ in
this $\mathrm{Spin}(7)$ is defined by the condition $A=B=C$, which implies that $A$ is an automorphism of
$\mathbb{O}$.

The group $\mathrm{Spin}(8)$ constructed above can  also be contained in a subgroup $\mathrm{Spin}(9)\subset \mathrm{SO}(15)$, as it can be seen as follows. Define
an inner product on
$\mathbb{R}^{16}=\mathcal{O}\oplus\mathcal{O}$ as the orthogonal product of the ones on both
$\mathcal{O}$-factors. Consider the octonian lines $L_m=\{(u,um)|u\in \mathbb{O} \}$ and
$L_\infty=\{(0,u)| u\in \mathbb{O}\}$. Then the group $\mathrm{Spin}(9)$ can be viewed as
\begin{equation}
\{ f\in \mathrm{SO}(16)|
f\, \mbox{maps octonian lines to octonian lines }\}.
\end{equation}
Notice that if we further require that $f(L_0)=L_0$ and $f(L_\infty)=L_\infty$, then it has the form
$f(u,v)=(Au,Cv)$. In this case,  it maps octonian lines to octonian lines. This implies that $A^{-1}(u)C(um)=A^{-1}(1)C(m)=B(m)$ for all $u$ and $m$, where $(A,B,C)$ is an arbitrary  triality triple. Thus the map  $(A,B,B)\to \mathrm{diag}(A,B)$ defines an injection from $\mathrm{Spin}(8)$ into $\mathrm{Spin}(9)$.

With all the groups involved defined as above, the sphere $S^7=\mathrm{Spin}(7)/G_2$ is the images of all $B(1)$, such that
 $(A,B,B)$ is a triality triple in $\mathrm{Spin}(7)$, and the sphere $S^{15}=\mathrm{Spin}(9)/\mathrm{Spin}(7)$ is the images
of all $f(1,0)$ with $f\in \mathrm{Spin}(9)$.

Through direct computation, it is not hard to determine the linear generators of the Lie algebras of
the groups $\mathrm{Spin}(7)$, $\mathrm{Spin}(8)$ and $\mathrm{Spin}(9)$ defined above.

\begin{lemma} \label{lie-alg-present}
Let $\{1, e_i, i=1,\ldots,7\}$ be an orthogonal basis of $\mathbb{O}$ and
  let the groups $\mathrm{Spin}(7)$, $\mathrm{Spin}(8)$ and $\mathrm{Spin}(9)$ be defined as above. Then we have:
\begin{description}
\item{\rm (1)}\quad The Lie algebra of $\mathrm{Spin}(7)$ is linearly generated by the set
$$S_1=\{2(E_{ji}-E_{ij}),L_{e_i}L_{e_j},L_{e_i}L_{e_j}, i<j\},$$
here  $E_{ji}-E_{ij}$ denotes the matrix in $\mathfrak{so}(7)\subset \mathfrak{so}(8)$ whose entries in the first column and the first row are all equal to zero.
\item{\rm (2)}\quad  The Lie algebra of $\mathrm{Spin}(8)$ is linearly generated by the set
 $$S_2=S_1\cup \{L_{e_i},R_{e_i},L_{e_i}+R_{e_i},\, \mbox{for all}\,\, i\}.$$
\item{\rm (3)}\quad  The Lie algebra of $\mathrm{Spin}(9)$ is linearly generated by the set
$$S_3=S_2\cup \left\{ \left(\begin{array}{cc}
 0 & 1 \\
  -1 & 0 \\
  \end{array}
  \right),
\left(\begin{array}{cc}
           0 & R_{e_i} \\
           R_{e_i} & 0 \\
         \end{array}
       \right)\right\},\,\mbox{for all}\,\, i\}.$$
\end{description}
\end{lemma}

\begin{proof}
(1)\quad
From the above easy facts on Cayley numbers, it can be easily checked that $(L_z R_{\overline{z}}, L_z, L_z)$
is a triality triple, for any unit length element $z\in \mathrm{Im}(\mathbb{O})$. Let $z=e_i$ and $z'=(\cos t) e_i+(\sin t) e_j$ for $i<j$.
Then the multiplication of the triality triples corresponding to them gives a one-parameter subgroup in $\mathrm{Spin}{7}$.
Differentiating them for all the pairs $i<j$ at $t=0$ gives the indicated linear basis of the Lie algebra of $\mathrm{Spin}(7)$.

(2)\quad For any unit length element $z\in\mathbb{O}$, there is a triality triple $(L_z, R_z, L_z R_z)$ in $\mathrm{Spin}(8)$. Differentiating them
for $z=\cos t+(\sin t) e_i$ at $t=0$, we get the  generators $\{L_{e_i},R_{e_i},L_{e_i}+R_{e_i},\, \mbox{for all}\,\, i\}$. Together with the generators of the Lie algebra
of $\mathrm{Spin}(7)$, they give a linear basis of the Lie algebra of $\mathrm{Spin}(8)$.

(3) For any $i$ and $t$, the elements $f_t(u,v)=((\cos t) u+(\sin t) v, -(\sin t) u+(\cos t) v)$
and $f_{t;i}=(\cos t u+ (\sin t) v e_i, (\cos t) u e_i +(\sin t) v)$ belong to $\mathrm{Spin}(9)$, i.e., they map octonian lines
to octonian lines. Differentiating $f_t$ and $f_{t;i}$ at $t=0$, we get the linear generators $\mathrm{Lie}\,(\mathrm{Spin}(8))$, $ \left(
                                                                             \begin{array}{cc}
                                                                               0 & 1 \\
                                                                               -1 & 0 \\
                                                                             \end{array}
                                                                           \right)
$ and $\left(
         \begin{array}{cc}
           0 & R_{e_i} \\
           R_{e_i} & 0 \\
         \end{array}
       \right)$ for the Lie algebra of $\mathrm{Spin}(9)$. Together with those elements in $S_2$   which are all block diagonal in $\mathfrak{so}(16)$, they
give a linear basis for the Lie algebra of $\mathrm{Spin}(9)$. This completes the proof of the proposition.
\end{proof}

\section{The coset spaces $S^7=\mathrm{Spin}(7)/G_2$ and $S^{15}=\mathrm{Spin}(9)/\mathrm{Spin}(7)$}
Assume that $X\in {\rm Lie}(\mathrm{Spin}(7))$ defines a Killing vector field of constant length $1$ on $S^7=\mathrm{Spin}(7)/G_2$ for a $\mathrm{Spin}(7)$-invariant Finsler metric $F$. By Lemma \ref{lie-alg-present}, $X$ can be
represented as a triple, whose first entry is a matrix in $\mathfrak{so}(7)$.
We arrange  the order of the basis vectors so that $e_1 e_2= e_3 e_4 = e_5 e_6 = e_7$.
Up to a suitable conjugation, we can write $X$ as $X=(X_1, X_2)$, where
\begin{eqnarray}
X_1 &=& 2a(E_{21}-E_{12})+2b(E_{43}-E_{34})+2c(E_{65}-E_{56}), \\
X_2 &=& a L_{e_1}L_{e_2}+bL_{e_3}L_{e_4}+cL_{e_5}L_{e_6},
aL_{e_1}L_{e_2}+bL_{e_3}L_{e_4}+cL_{e_5}L_{e_6}.
\end{eqnarray}
The projection of $X$ to $\mathfrak{m}$ is $a e_1 e_2+ b e_3 e_4 + c e_5 e_6=(a+b+c)e_7$.
The Weyl group action can arbitrarily change the signs of $a$, $b$ and $c$. So the projection of
the orbit $\mathcal{O}_X$ to $\mathfrak{m}$ contains eight points $(\pm a\pm b\pm c)e_7$, all of them lying in the same line.
So the set $\{\pm a, \pm b, \pm c\} $ contains  at most two elements. This implies that at least two of $a$, $b$ and $c$ must be zero.

Suppose two of $a$, $b$ and $c$ are zero. We assert that the metric $F$ is  a symmetric Riemannian metric.
Without loss of generality, we can simplify $X$ to $(2(E_{21}-E_{12},L_{e_1}L_{e_2},L_{e_1}L_{e_2})$. Notice that
for any  element $z\in\mathrm{Im}\mathbb{O}$ with unit length, one can find $z_1, z_2\in\mathrm{Im} \mathbb{O}$, such that $z=z_1 z_2$ and
$\langle z_1,z_2\rangle=0$. Hence in the conjugation class of $X$, one can find a vector $X'$ whose second entry is $L_{z_1}L_{z_2}$.
This  indicates that $z=\mathrm{pr}X'$ is also a unit vector for $F$. Thus $F$ is a symmetric Riemannian  metric.

In summarizing, we have

\begin{proposition}\label{last1}
Let $F$ be a $\mathrm{Spin}(7)$-invariant   Finsler metric on the sphere $S^7=\mathrm{Spin}(7)/G_2$. If there is
a nonzero vector $X\in \mathfrak{so}(7)=\mathrm{Lie}(\mathrm{Spin}(7))$ defining a Killing vector field of constant length of $F$,
then $F$ is a  symmetric Riemannian metric.
\end{proposition}

Assume that $X\in\mathrm{Lie}(\mathrm{Spin}(9))$ defines a Killing vector field of constant length $1$ for a $\mathrm{Spin}(9)$-invariant
Finsler metric on the sphere $S^{15}=\mathrm{Spin}(9)/\mathrm{Spin}(7)$. Up to a suitable conjugation, we can assume that $X$
belongs to the subalgebra $\mathrm{Lie}(\mathrm{Spin}(8))$, i.e.,  $X$  is a block diagonal $\mathrm{diag}(A,C)\in \mathfrak{so}(15)$,
or can be represented as a triple $(A,B,C)$. With the submanifold metric, the sphere $S^7=\mathrm{Spin}(8)/\mathrm{Spin}(7)\subset \mathrm{Spin}(9)/\mathrm{Spin}(7)=S^{15}$
has a connected isometric group $\mathrm{SO}(8)$. Thus the above $S^7$ with the submanifold metric is   a symmetric Riemannian metric. The restriction of the vector field $X$  to $S^7$ is a Killing vector field of constant length $1$ with respect to the submanifold metric. Thus by a suitable conjugation and rescaling of the metric, we can assume that the first
entry of $X$ is $L_{e_1}$. Then the triple for $X$ is
$(L_{e_1},R_{e_1},L_{e_1}+R_{e_1})$. Its projection to $\mathfrak{m}$ is $(e_1,0)$. Then its conjugation
by $f_{\pi/2}=\left(
                                                                             \begin{array}{cc}
                                                                               0 & 1 \\
                                                                               -1 & 0 \\
                                                                             \end{array}
                                                                           \right)$, i.e.,
\begin{equation}
f_{\pi/2}X f_{\pi/2}^{-1}=
\left(
                                                                             \begin{array}{cc}
                                                                               0 & 1 \\
                                                                               -1 & 0 \\
                                                                             \end{array}
                                                                           \right)
                                                                           \left(
                                                                             \begin{array}{cc}
                                                                               L_{e_1} & 0 \\
                                                                               0 & L_{e_1}+R_{e_1} \\
                                                                             \end{array}
                                                                           \right)
                                                                           \left(
                                                                             \begin{array}{cc}
                                                                               0 & -1 \\
                                                                               1 & 0 \\
                                                                             \end{array}
                                                                           \right)=
                                                                          \left(
                                                                             \begin{array}{cc}
                                                                               L_{e_1}+R_{e_1} & 0 \\
                                                                               0 & L_{e_1} \\
                                                                             \end{array}
                                                                           \right),\nonumber
\end{equation}
has a projection $2e_1$ in $\mathfrak{m}$. This is  impossible when $X$ is of constant length.

In summarizing, we have the following proposition.

\begin{proposition}\label{last2}
Let $F$ be a $\mathrm{Spin}(9)$-invariant  Finsler metric on the sphere $S^{15}=\mathrm{Spin}(9)/\mathrm{Spin}(7)$. Suppose $X\in \mathrm{Lie}(\mathrm{Spin}(9))$
 defines a Killing field of constant length of $F$. Then $X$ must be zero.
\end{proposition}

Now the main theorem follows from Propositions \ref{prop0}, \ref{last1} and \ref{last2}.

\vspace{8mm}
\noindent {\bf Acknowledgement. }\quad We are deeply grateful to the referee for a  careful reading of the manuscript and valuable suggestions. This work was finished during the first author's visit to the Chern institute of Mathematics. He would like to express his deep gratitude to the members of the institute for their hospitality.

\end{document}